\documentclass[10pt,english]{amsart}
\usepackage{amsfonts, amssymb, amsmath, amsthm,eucal,latexsym,nicefrac,mathrsfs}
\usepackage[varg]{txfonts}

\usepackage[english]{babel}

\usepackage{pstricks}
\usepackage{graphics} 
\usepackage{psfrag}  
\usepackage[all]{xy}
\DeclareMathAlphabet{\mathpzc}{OT1}{pzc}{m}{it}

\theoremstyle{plain}
\newtheorem{theorem}{Theorem}[section]
\newtheorem{lemma}[theorem]{Lemma}
\newtheorem{theo}[theorem]{Theorem}

\newtheorem{coro}[theorem]{Corollary}
\newtheorem{prop}[theorem]{Proposition}

\theoremstyle{definition}

\theoremstyle{remark}
\newtheorem{rema}[theorem]{Remark}

\def\g{{\mathfrak{g}}}  

\def\k{{\Bbbk}}

\def\rg{\ell}               
\def\r{{\rm reg}}

\def\poie#1#2#3#4#5#6#7#8#9{\def\un{#5#6#7#8#9}\def\deux{#6#7#8#9}\def\trois{#2#4#8#9}
\def\quatre{#8#9}\def\cinq{#5#6#7}\def\six{#6#7}\def\sept{#2#4}
\ifx\un\empty {#1}_{#2}{#3 \hskip 0.15em}{#1}_{#4} \else \ifx\deux\empty 
{#5}(#1_{#2}){#3 \hskip 0.15em}{#5}(#1_{#4})
\else \ifx\trois\empty {#5}_{#6}(#1){#3 \hskip 0.15em}{#5}_{#7}(#1) 
\else \ifx\quatre\empty {#5}_{#6}(#1_#2){#3 \hskip 0.15em}{#5}_{#7}(#1_#4) 
\else \ifx\cinq\empty {#1}_{#2}^{#8}{#3 \hskip 0.15em}#1_#4^{#9} 
\else \ifx\six\empty {#5}(#1_{#2}^{#8}){#3 \hskip 0.15em}{#5}(#1_{#4}^{#9}) 
\else \ifx\sept\empty {#5}_{#6}(#1)^{#8}{#3 \hskip 0.15em}{#5}_{#7}(#1)^{#9} \else
{#5}_{#6}(#1_{#2}^{#8})^{#9}{#3 \hskip 0.15em}{#5}_{#7}(#1_{#4}^{#8})^{#9} 
\fi \fi \fi \fi \fi \fi \fi}
\def\poi#1#2#3#4#5#6#7{\def\un{#5#6#7}\def\deux{#6#7}
\def\trois{#2#4} \def\cinq{#3#4#5}
\ifx\un\empty {#1}_{#2}{#3 \hskip 0.15em}{#1}_{#4} \else
\ifx\deux\empty {#5}(#1_{#2}){#3 \hskip 0.15em}{#5}(#1_{#4}) \else
\ifx\trois\empty {#5}_{#6}(#1){#3 \hskip 0.15em}{#5}_{#7}(#1) \else
{#5_{#6}}(#1_{#2}){#3 \hskip 0.15em}{#5_{#7}}(#1_{#4}) \fi \fi \fi}
\def\rond{\raisebox{.3mm}{\scriptsize$\circ$}}

\def\tens{\raisebox{.3mm}{\scriptsize$\otimes$}}

\def\dv#1#2{\langle {#1},{#2}\rangle}

\def\tk#1#2{{#2}\otimes _{#1}}

\def\no{n$^{\circ}$}
\def\hd{{\mathrm {hd}}\hskip 0.15em}
\def\ec#1#2#3#4#5{\def\un{#3#4#5}\def\deux{#3#5}\def\trois{#3}
\def\four{#2#4#5}\def\five{#2#5}\def\six{#2}\def\seven{#3#4}
\def\eight{#2#4} \def\nine{#2#3#4}
\ifx\nine\empty {\rm #1}_{#5} \else
\ifx\un\empty {\rm #1}({\goth #2}) \else
\ifx\deux\empty {\rm #1}({\goth #2}_{#4}) \else
\ifx\trois\empty {\rm #1}_{#5}({\goth #2}_{#4}) \else
\ifx\four\empty {\rm #1}(#3) \else
\ifx\five\empty {\rm #1}(#3_{#4}) \else
\ifx\six\empty {\rm #1}_{#5}(#3_{#4}) \else
\ifx\seven\empty {\rm #1}_{#5} ({\goth#2})\else
\ifx\eight\empty {\rm #1}_{#5}({#3})
\fi \fi \fi \fi \fi \fi \fi \fi \fi}
\def\hec#1#2#3#4#5{\def\un{#3#4#5}\def\deux{#3#5}\def\trois{#3}
\def\four{#2#4#5}\def\five{#2#5}\def\six{#2}\def\seven{#3#4}
\def\eight{#2#4} \def\nine{#2#3#4}
\ifx\nine\empty \hat{{\rm #1}}_{#5} \else
\ifx\un\empty \hat{{\rm #1}}({\goth #2}) \else
\ifx\deux\empty \hat{{\rm #1}}({\goth #2}_{#4}) \else
\ifx\trois\empty \hat{{\rm #1}}_{#5}({\goth #2}_{#4}) \else
\ifx\four\empty \hat{{\rm #1}}(#3) \else
\ifx\five\empty \hat{{\rm #1}}(#3_{#4}) \else
\ifx\six\empty \hat{{\rm #1}}_{#5}(#3_{#4}) \else
\ifx\seven\empty \hat{{\rm #1}}_{#5} ({\goth#2})  \else
\ifx\eight\empty \hat{{\rm #1}}_{#5}({#3})
\fi \fi \fi \fi \fi \fi \fi \fi \fi}
\def\e#1#2{\ec {#1}#2{}{}{}}
\def\es#1#2{\ec {#1}{}{#2}{}{}}

\def\ai#1#2#3{\def\deux{#2#3} \def\trois{#3} \def\quatre{#2} 
\ifx\deux\empty \es S{{\goth #1}}^{{\goth #1}} \else
\ifx\trois\empty \es S{{\goth #1}^{#2}}^{{\goth #1}^{#2}} \else
\ifx\quatre\empty \es S{{\goth #1}_{#3}}^{{\goth #1}_{#3}} \else
\es S{{\goth #1}_{#3}^{#2}}^{{\goth #1}_{#3}^{#2}} \fi \fi \fi}

\def\Bbb{\mathbb}
\def\goth{\mathfrak}
\def\cal{\mathcal}

\def\gi#1#2#3#4{\def\trois{#3#4} \def\quatre{#4}\def\cinq{#3}\ifx\trois\empty {\rm i}_{#1,{\goth #2}}
\else \ifx\quatre\empty {\rm i}_{#1_{#3},{\goth #2}} \else\ifx\cinq\empty {\rm i}_{#1,{\goth #2}_{#4}} \else {\rm i}_{#1_{#3},{\goth #2}_{#4}} \fi \fi \fi}
\def\j#1#2{\def\deux{#2} \ifx\deux\empty {\rm rk}\hskip .125em{{\goth #1}} \else {\rm rk}\hskip .125em{{\goth #1}_{#2}} \fi}
\def\aj#1#2{\def\deux{#2} \ifx\deux\empty {\rm j}_{{\goth #1}} \else {\rm j}_{{\goth #1}_{#2}} \fi}
\def\an#1#2{\def\deux{#2} \ifx\deux\empty {\cal O}_{#1} \else {\cal O}_{#1,#2} \fi }
\def\han#1#2{\def\deux{#2} \ifx\deux\empty {\hat{{\cal O}}}_{#1} \else {\hat{{\cal O}}}_{#1,#2} \fi }

\def\gg#1#2{{\goth #1}_{#2}\times {\goth #1}_{#2}}

\def\sgg#1#2{\es S{\gg {#1}{#2}}}

\def\dim{{\rm dim}\hskip .125em}
\def\dd{{\rm d}}

\def\j#1#2{\ell _{{\goth #1}_{#2}}}

\def\ex #1#2{\mbox{$\bigwedge^{#1}(#2)$}} 

\def\b#1#2{{\mathrm {b}}_{{\mathfrak{#1}}_{#2}}}

\def\bi#1#2{{\mathrm B}_{{\goth {#1}_{#2}}}}

\def\bj#1#2{{\mathrm A}_{{\goth {#1}}_{#2}}}
\def\bk#1#2{{\mathrm C}_{{\goth #1}_{#2}}}

\def\bkk{{\mathrm C}}

\title
[Commuting variety]
{On the Commuting variety of a reductive Lie algebra}

\author
[J-Y Charbonnel]{Jean-Yves Charbonnel}

\address{Jean-Yves Charbonnel, Universit\'e Paris Diderot - CNRS \\
Institut de Math\'ematiques de Jussieu - Paris Rive Gauche\\
UMR 7586 \\ Groupes, repr\'esentations et g\'eom\'etrie \\
B\^atiment Sophie Germain \\ Case 7012 \\ 
75205 Paris Cedex 13, France}
\email{jean-yves.charbonnel@imj-prg.fr}



\date\today

\begin{document}

\large

\begin{abstract}
The commuting variety of a reductive Lie algebra ${\goth g}$ is the underlying 
variety of a well defined subscheme of $\gg g{}$. In this note, it is proved that 
this scheme is normal. In particular, its ideal of definition is a prime ideal. 
\end{abstract}

\maketitle

\setcounter{tocdepth}{1}
\tableofcontents

\section{Introduction} \label{int}
In this note, the base field $\k$ is algebraically closed of characteristic $0$, 
${\goth g}$ is a reductive Lie algebra of finite dimension, $\rg$ is its rank,
and $G$ is its adjoint group.    

\subsection{} \label{int1}
The dual of ${\goth g}$ identifies to ${\goth g}$ by a non degenerate symmetric bilinear 
form on ${\goth g}$ extending the Killing form of the derived  algebra of ${\goth g}$.
Denote by $(v,w)\mapsto \dv vw$ this bilinear form and denote by $I_{{\goth g}}$ the 
ideal of $\k[\gg g{}]$ generated by the functions $(x,y)\mapsto \dv v{[x,y]}$'s with $v$ 
in ${\goth g}$. The commuting variety ${\cal C}({\goth g})$ of ${\goth g}$ is the 
subvariety of elements $(x,y)$ of $\gg g{}$ such that $[x,y]=0$. It is the underlying 
variety to the subscheme ${\cal S}({\goth g})$ of $\gg g{}$ defined by $I_{{\goth g}}$. 
It is a well known and long standing open question whether or not this scheme is reduced,
that is ${\cal C}({\goth g})={\cal S}({\goth g})$. According to Richardson~\cite{Ric}, 
${\cal C}({\goth g})$ is irreducible and according to Popov~\cite[Theorem 1]{Po}, the 
singular locus of ${\cal S}({\goth g})$ has codimension at least $2$ in 
${\cal C}({\goth g})$. Then, according to Serre's normality criterion, arises the 
question to know whether or not ${\cal C}({\goth g})$ is normal. There are 
many results about the commuting variety. A result of Dixmier~\cite{Di2} proves that 
$I_{{\goth g}}$ contains all the elements of the radical of $I_{{\goth g}}$ which 
have degree $1$ in the second variable. In \cite{Ga}, Gan and Ginzburg prove that for 
${\goth g}$ simple of type ${\mathrm {A}}$, the invariant elements under $G$ of 
$I_{{\goth g}}$ is a radical ideal of the algebra $\k[\gg g{}]^{G}$ of invariant elements
of $\k[\gg g{}]$ under $G$. In \cite{Gi}, Ginzburg proves that the normalisation of 
${\cal C}({\goth g})$ is Cohen-Macaulay.  

\subsection{Main results and sketch of proofs.} \label{int2}    
According to the identfication of ${\goth g}$ and its dual, $\k[\gg g{}]$ equals the 
symmetric algebra $\sgg g{}$ of $\gg g{}$. The main result of this note is the following 
theorem:

\begin{theo}\label{tint}
The subscheme of $\gg g{}$ defined by $I_{{\goth g}}$ is Cohen-Macaulay and normal. 
Furthermore, $I_{{\goth g}}$ is a prime ideal of $\sgg g{}$.
\end{theo}
 
According to Richardson's result and Popov's result, it suffices to prove that the scheme 
${\cal S}({\goth g})$ is Cohen-Macaulay. The main idea of the proof in the theorem
uses the main argument of the Dixmier's proof: for a finitely generated module $M$ over 
$\sgg g{}$, $M=0$ if the codimension of its support is at least $l+2$ with $l$ the 
projective dimension of $M$ (see Appendix~\ref{p}). 

Introduce the characteristic submodule of ${\goth g}$, denoted by $\bi g{}$. By 
definition, $\bi g{}$ is a submodule of $\tk {\k}{\sgg g{}}{\goth g}$ and an element 
$\varphi $ of $\tk {\k}{\sgg g{}}{\goth g}$ is in $\bi g{}$ if and only if for all 
$(x,y)$ in $\gg g{}$, $\varphi (x,y)$ is in the sum of subspaces ${\goth g}^{ax+by}$ with
$(a,b)$ in $\k^{2}\setminus \{0\}$ and ${\goth g}^{ax+by}$ the centralizer of $ax+by$ in 
${\goth g}$. According to a Bolsinov's result, $\bi g{}$ is a free $\sgg g{}$-module of 
rank $\b g{}$, the dimension of the Borel subalgebras of ${\goth g}$. Moreover, the 
orthogonal complement of $\bi g{}$ in $\tk {\k}{\sgg g{}}{\goth g}$ is a free 
$\sgg g{}$-module of rank $\b g{}-\rg$. These two results are fundamental in the proof
of the following proposition:

\begin{prop}\label{pint}
For $i$ positive integer, the submodule $\ex i{{\goth g}}\wedge \ex {\b g{}}{\bi g{}}$ 
of $\tk {\k}{\sgg g{}}\ex {i+\b g{}}{{\goth g}}$ has projective dimension at most $i$.
\end{prop}

Denoting by $E$ the quotient of $\tk {\k}{\sgg g{}}{\goth g}$ by $\bi g{}$, let $E_{i}$
be the quotient of $\ex iE$ by its torsion module. The $\sgg g{}$-modules 
$\ex i{{\goth g}}\wedge \ex {\b g{}}{\bi g{}}$ and $E_{i}$ are isomorphic. Furthermore,
for $i\geq 2$, $E_{i}$ is isomorphic to a direct factor of the quotient of 
$\tk {\sgg g{}}EE_{i-1}$ by its torsion module. Denoting by $\overline{E_{i-1}}$ this 
quotient, the projective dimension of $\overline{E_{i-1}}$ is at most $d_{i-1}+1$ if 
$d_{i-1}$ is the projective dimension of $E_{i-1}$, whence a proof of the proposition by 
induction on $i$.

Let $\dd$ be the $\sgg g{}$-derivation of the algebra 
$\tk {\k}{\sgg g{}}\ex {}{{\goth g}}$ such that for all $v$ in ${\goth g}$, $\dd v $ is 
the function on $\gg g{}$, $(x,y)\mapsto \dv v{[x,y]}$. Then the ideal of 
$\tk {\k}{\sgg g{}}\ex {}{{\goth g}}$ generated by $\ex {\b g{}}{\bi g{}}$ is a graded
subcomplex of the graded complex $\tk {\k}{\sgg g{}}\ex {}{{\goth g}}$. The support
of the homology of this complex is contained in ${\cal C}({\goth g})$. Then we deduce 
from Proposition~\ref{pint} that this complex has no homology in degree different 
from $\b g{}$ and $\k[{\cal C}({\goth g})]$ is Cohen-Macaulay by Auslander-Buchsbaum's 
theorem.  

\subsection{Notations}\label{int3}
$\bullet$ For $V$ a module over a $\k$-algebra, its symmetric and exterior algebras are 
denoted by $\ec S{}V{}{}$ and $\ex {}V$ respectively. If $E$ is a 
subset of $V$, the submodule of $V$ generated by $E$ is denoted by span($E$). When $V$
is a vector space over $\k$, the grassmannian of all $d$-dimensional subspaces of $V$ is 
denoted by Gr$_d(V)$.

$\bullet$
All topological terms refer to the Zariski topology. If $Y$ is a subset of a topological
space $X$, denote by $\overline{Y}$ the closure of $Y$ in $X$. For $Y$ an open subset
of the algebraic variety $X$, $Y$ is called {\it a big open subset} if the codimension
of $X\setminus Y$ in $X$ is at least $2$. For $Y$ a closed subset of an algebraic 
variety $X$, its dimension is the biggest dimension of its irreducible components and its
codimension in $X$ is the smallest codimension in $X$ of its irreducible components. For 
$X$ an algebraic variety, $\k[X]$ is the algebra of regular functions on $X$.

$\bullet$
All the complexes considered in this note are graded complexes over ${\Bbb Z}$
of vector spaces and their differentials are homogeneous of degree $-1$ and they are 
denoted by $\dd $. As usual, the gradation of the complex is denoted by $C_{\bullet}$. 

$\bullet$ The dimension of the Borel subalgebras of ${\goth g}$ is denoted by $\b g{}$.
Set $n := \b g{}-\rg$ so that $\dim {\goth g}= 2\b g{} - \j g{}=2n+\rg$.

$\bullet$
The dual ${\goth g}^{*}$ of ${\goth g}$ identifies with ${\goth g}$ by a given non 
degenerate, invariant, symmetric bilinear form $\dv ..$ on $\gg g{}$
extending the Killing form of $[{\goth g},{\goth g}]$. 

$\bullet$ 
For $x \in \g$, denote by ${\goth g}^{x}$ the centralizer of $x$ in 
${\goth g}$. The set of regular elements of $\g$ is 
$$\g_{\r} \ := \ \{ x\in \g \ \vert \ \dim \g^x=\rg \} .$$
The subset $\g_{\r}$ of ${\goth g}$ is a $G$-invariant open subset of ${\goth g}$.
According to \cite{Ve}, ${\goth g}\setminus {\goth g}_{\r}$ is equidimensional of 
codimension $3$.  

$\bullet$
Denote by $\e Sg^{{\goth g}}$ the algebra of ${\goth g}$-invariant elements of 
$\e Sg$. Let $p_1,\ldots,p_{\rg}$ be homogeneous generators of $\e Sg^{\g}$ of degree
$\poi d1{,\ldots,}{\rg}{}{}{}$ respectively. Choose the polynomials
$\poi p1{,\ldots,}{\rg}{}{}{}$ so that $\poi d1{\leq\cdots \leq }{\rg}{}{}{}$. For 
$i=1,\ldots,\rg$ and $(x,y)\in\g \times \g$, consider a shift of $p_i$ in 
direction $y$: $p_i(x+ty)$ with $t\in\k$. Expanding $p_i(x+ty)$ as a polynomial in 
$t$, one obtains
\begin{eqnarray}\label{eq:pi}
p_i(x+ty)=\sum\limits_{m=0}^{d_i} p_{i}^{(m)} (x,y) t^m;  && \forall
(t,x,y)\in\k\times\g\times\g
\end{eqnarray}
where $y \mapsto (m!)p_{i}^{(m)}(x,y)$ is the derivative at $x$ of $p_i$ at the order
$m$ in the direction $y$. The elements $p_{i}^{(m)}$ defined by~(\ref{eq:pi}) are
invariant elements of $\tk {\k}{\e Sg}\e Sg$ under the diagonal action of $G$ in
$\gg g{}$. Remark that $p_i^{(0)}(x,y)=p_i(x)$ while $p_i^{(d_i)}(x,y)=p_i(y)$  for 
all $(x,y)\in \g\times \g$.

\begin{rema}\label{rint}
The family
$\mathcal{P}_x  :=
\{p_{i}^{(m)}(x,.); \ 1 \leq i \leq \rg, 1 \leq m \leq d_i  \}$ for $x\in\g$,
is a Poisson-commutative family of $\e Sg$ by Mishchenko-Fomenko~\cite{MF}.
One says that the family $\mathcal{P}_x$ is constructed by
the \emph{argument shift method}.
\end{rema}

$\bullet$
Let $i \in\{1,\ldots,\rg\}$. For $x$ in $\g$, denote by $\varepsilon _i(x)$ the 
element of $\g$ given by
$$ \dv {\varepsilon _{i}(x)}y = \frac{\dd }{\dd t} p_{i}(x+ty) \left \vert _{t=0} \right.
$$
for all $y$ in ${\goth g}$. Thereby, $\varepsilon _{i}$ is an invariant element of 
$\tk {\k}{\e Sg}\g$ under the canonical action of $G$. According to 
\cite[Theorem 9]{Ko}, for $x$ in ${\goth g}$, $x$ is in ${\goth g}_{\r}$ if and only if 
$\poi x{}{,\ldots,}{}{\varepsilon }{1}{\rg}$ are linearly independent. In this case, 
$\poi x{}{,\ldots,}{}{\varepsilon }{1}{\rg}$ is a basis of ${\goth g}^{x}$. 

Denote by $\varepsilon _{i}^{(m)}$, for $0\leq m\leq d_i-1$, the elements of 
$\tk{\k}{\sgg g{}}{\goth g}$ defined by the equality:
\begin{eqnarray}\label{eq:phi}
\varepsilon _i (x+ty) =\sum\limits_{ m=0}^{d_i-1} \varepsilon _i^{(m)}(x,y) t^m , &&
\forall (t,x,y)\in\k\times\g\times\g
\end{eqnarray}
and set:
$$ V_{x,y} := 
{\mathrm {span}}(\{\poie {x,y}{}{,\ldots,}{}{\varepsilon }{i}{i}{(0)}{(d_{i}-1)}, \ 
i =1,\ldots,\rg\}) $$ 
for $(x,y)$ in $\gg g{}$. 

\section{Characteristic module} \label{sc}
For $(x,y)$ in $\gg g{}$, set:
$$ V'_{x,y} = \sum_{(a,b) \in \k^{2}\setminus \{0\}} {\goth g}^{ax+by} ,$$
and denote by $P_{x,y}$ the span of $x$ and $y$. By definition, the characteristic module
$\bi g{}$ of ${\goth g}$ is the submodule of elements $\varphi $ of 
$\tk {\k}{\sgg g{}}{\goth g}$ such that $\varphi (x,y)$ is in $V'_{x,y}$ for all $(x,y)$ 
in $\gg g{}$. In this section, some properties of $\bi g{}$ are given.

\subsection{} \label{sc1}
Denote by $\Omega _{{\goth g}}$ the subset of elements $(x,y)$ of $\gg g{}$ such that
$P_{x,y}$ has dimension $2$ and such that $P_{x,y}\setminus \{0\}$ is contained in 
${\goth g}_{\r}$. According to~\cite[Corollary 10]{CMo}, $\Omega _{{\goth g}}$ is a big 
open subset of $\gg g{}$. 

\begin{prop}\label{psc1}
Let $(x,y)$ be in $\gg g{}$ such that $P_{x,y}\cap {\goth g}_{\r}$ is not empty.  

{\rm (i)} Let $O$ be an open subset of $\k^{2}$ such that $ax+by$ is in ${\goth g}_{\r}$ 
for all $(a,b)$ in $O$. Then $V_{x,y}$ is the sum of the ${\goth g}^{ax+by}$'s, 
$(a,b) \in O$.

{\rm (ii)} The spaces $[x,V_{x,y}]$ and $[y,V_{x,y}]$ are equal.

{\rm (iii)} The space $V_{x,y}$ has dimension at most $\b g{}$ and the equality holds if 
and only if $(x,y)$ is in $\Omega _{{\goth g}}$.

{\rm (iv)} The space $[x,V_{x,y}]$ is orthogonal to $V_{x,y}$. Furthermore, $(x,y)$ is 
in $\Omega _{{\goth g}}$ if and only if $[x,V_{x,y}]$ is the orthogonal complement of 
$V_{x,y}$ in ${\goth g}$.

{\rm (v)} The space $V_{x,y}$ is contained in $V'_{x,y}$. Moreover, $V_{x,y}=V'_{x,y}$ 
if $(x,y)$ is in $\Omega _{{\goth g}}$.

{\rm (vi)} For $i=1,\ldots,\rg$ and for $m=0,\ldots,d_{i}-1$, $\varepsilon _{i}^{(m)}$ is
a $G$-equivariant map.
\end{prop}

\begin{proof}
(i) For pairwise different elements $\poi t{i,1}{,\ldots,}{i,d_{i}-1}{}{}{}$, 
$i=1,\ldots,\rg$ of $\k\setminus \{0\}$, the $\varepsilon _{i}^{(m)}(x,y)$'s, 
$m=0,\ldots,d_{i}-1$ are linear combinations of the $\varepsilon _{i}(x+t_{i,j}y)$'s, 
$j=1,\ldots,d_{i}-1$ for $i=1,\ldots,\rg$. Furthermore, for all $z$ in ${\goth g}_{\r}$, 
$\poi z{}{,\ldots,}{}{\varepsilon }{1}{\rg}$ is a basis of ${\goth g}^{z}$ by 
~\cite[Theorem 9]{Ko}, whence the assertion since the maps 
$\poi {\varepsilon }1{,\ldots,}{\rg}{}{}{}$ are homogeneous.

(ii) Let $O$ be an open subset of $(\k\setminus \{0\})^{2}$ such that 
$ax+by$ is in ${\goth g}_{\r}$ for all $(a,b)$ in $O$. For all $(a,b)$ in $O$, 
$[x,{\goth g}^{ax+by}]=[y,{\goth g}^{ax+by}]$ since $[ax+by,{\goth g}^{ax+by}]=0$ and 
since $ab\neq 0$, whence the assertion by (i).  

(iii) According to~\cite[Ch. V, \S 5, Proposition 3]{Bou}, 
$$ \poi d1{+\cdots +}{\rg}{}{}{} = \b g{} .$$
So $V_{x,y}$ has dimension at most $\b g{}$. By \cite[Theorem 2.1]{Bol}, $V_{x,y}$
has dimension $\b g{}$ if and only if $(x,y)$ is in $\Omega _{{\goth g}}$.

(iv) According to \cite[Theorem 2.1]{Bol}, $V_{x,y}$ is a totally isotropic subspace 
with respect to the skew bilinear form on ${\goth g}$
$$ (v,w) \longmapsto \dv {ax+by}{[v,w]}$$
for all $(a,b)$ in $\k^{2}$. As a result, by invariance of $\dv ..$, $V_{x,y}$ is 
orthogonal to $[x,V_{x,y}]$. If $(x,y)$ is in $\Omega _{{\goth g}}$, ${\goth g}^{x}$ has 
dimension $\rg$ and it is contained in $V_{x,y}$. Hence, by (iii),
$$ \dim [x,V_{x,y}] = \b g{}-\rg = \dim {\goth g} - \dim V_{x,y}$$
so that $[x,V_{x,y}]$ is the orthogonal complement of $V_{x,y}$ in ${\goth g}$. 
Conversely, if $[x,V_{x,y}]$ is the orhogonal complement of $V_{x,y}$ in ${\goth g}$, 
then
$$ \dim V_{x,y} + \dim [x,V_{x,y}] = \dim {\goth g} .$$
Since $P_{x,y}\cap {\goth g}_{\r}$ is not empty, ${\goth g}^{ax+by}\cap V_{x,y}$ has 
dimension $\rg$ for all $(a,b)$ in a dense open subset of $\k^{2}$. By continuity,
${\goth g}^{x}\cap V_{x,y}$ has dimension at least $\rg$ so that 
$$ 2\dim V_{x,y} - \rg \geq \dim {\goth g} .$$
Hence, by (iii), $(x,y)$ is in $\Omega _{{\goth g}}$.

(v) According to \cite[Theorem 9]{Ko}, for all $z$ in ${\goth g}$ and for 
$i=1,\ldots,\rg$, $\varepsilon _{i}(z)$ is in ${\goth g}^{z}$. Hence for all $t$ in 
$\k$, $\varepsilon _{i}(x+ty)$ is in $V'_{x,y}$. So 
$\varepsilon _{i}^{(m)}(x,y)$ is in $V'_{x,y}$ for all $m$, whence 
$V_{x,y}\subset V'_{x,y}$.

Suppose that $(x,y)$ is in $\Omega _{{\goth g}}$. According to \cite[Theorem 9]{Ko},
for all $(a,b)$ in $\k ^{2}\setminus \{0\}$, 
$\poi {ax+by}{}{,\ldots,}{}{\varepsilon }{1}{\rg}$ is a basis of ${\goth g}^{ax+by}$. 
Hence ${\goth g}^{ax+by}$ is contained in $V_{x,y}$, whence the assertion. 

(vi) Let $i$ be in $\{1,\ldots,\rg\}$. Since $p_{i}$ is $G$-invariant, 
$\varepsilon _{i}$ is a $G$-equivariant map. As a result, its $2$-polarizations 
$\poie {\varepsilon }i{,\ldots,}{i}{}{}{}{(0)}{(d_{i}-1)}$ are $G$-equivariant under the
diagonal action of $G$ in $\gg g{}$. 
\end{proof}

\begin{theo}\label{tsc1}
{\rm (i)} The module $\bi g{}$ is a free module of rank $\b g{}$ whose a basis is the 
sequence $\varepsilon _{i}^{(0)},\ldots,\varepsilon _{i}^{(d_{i}-1)}$, $i=1,\ldots,\rg$. 

{\rm (ii)} For $\varphi $ in $\tk {\k}{\sgg g{}}{\goth g}$, $\varphi $ is in 
$\bi g{}$ if and only if $p\varphi \in \bi g{}$ for some $p$ in $\sgg g{}\setminus \{0\}$.

{\rm (iii)} For all $\varphi $ in $\bi g{}$ and for all $(x,y)$ in $\gg g{}$, 
$\varphi (x,y)$ is orthogonal to $[x,y]$.
\end{theo}

\begin{proof}
(i) and (ii) According to Proposition~\ref{psc1},(v), $\varepsilon _{i}^{(m)}$ is in 
$\bi g{}$ for all $(i,m)$. Moreover, according to Proposition~\ref{psc1},(iii), these 
elements are linearly independent over $\sgg g{}$. Let $\varphi $ be an element of 
$\tk {\k}{\sgg g{}}{\goth g}$ such that $p\varphi $ is in $\bi g{}$ for some $p$
in $\sgg g{}\setminus \{0\}$. Since $\Omega _{{\goth g}}$ is a big open subset of 
$\gg g{}$, for all $(x,y)$ in a dense open subset of $\Omega _{{\goth g}}$, 
$\varphi (x,y)$ is in $V_{x,y}$ by Proposition~\ref{psc1},(v). According to 
Proposition~\ref{psc1},(iii), the map
$$\Omega _{{\goth g}} \longrightarrow \ec {Gr}g{}{}{\b g{}}, \qquad
(x,y) \longmapsto V_{x,y} $$
is regular. So, $\varphi (x,y)$ is in $V_{x,y}$ for all $(x,y)$ in $\Omega _{{\goth g}}$
and for some regular functions $a_{i,m}$, $i=1,\ldots,\rg$, $m=0,\ldots,d_{i}-1$ on 
$\Omega _{{\goth g}}$,
$$ \varphi (x,y) = \sum_{i=1}^{\rg} \sum_{m=0}^{d_{i}-1} 
a_{i,m}(x,y)\varepsilon _{i}^{(m)}(x,y) $$
for all $(x,y)$ in $\Omega _{{\goth g}}$. Since $\Omega _{{\goth g}}$ is a big open subset
of $\gg g{}$ and since $\gg g{}$ is normal, the $a_{i,m}$'s have a regular extension to 
$\gg g{}$. Hence $\varphi $ is a linear combination of the 
$\varepsilon _{i}^{(m)}$'s with coefficients in $\sgg g{}$. As a result, the sequence 
$\varepsilon _{i}^{(m)}$, $i=1,\ldots,\rg$, $m=0,\ldots,d_{i}-1$ is a basis of the module
$\bi g{}$ and $\bi g{}$ is the subset of elements $\varphi $ of 
$\tk {\k}{\sgg g{}}{\goth g}$ such that $p\varphi \in \bi g{}$ for some $p$ in 
$\sgg g{}\setminus \{0\}$. 

(iii) Let $\varphi $ be in $\bi g{}$. According to (i) and Proposition~\ref{psc1},(iv),
for all $(x,y)$ in $\Omega _{{\goth g}}$, $[x,\varphi (x,y)]$ is orthogonal 
to $V_{x,y}$. Then, since $y$ is in $V_{x,y}$, $[x,\varphi (x,y)]$ is orthogonal to 
$y$ and $\dv {\varphi (x,y)}{[x,y]}=0$, whence the assertion.
\end{proof}

\subsection{} \label{sc2}
Also denote by $\dv ..$ the natural extension of $\dv ..$ to the module 
$\tk {\k}{\sgg g{}}{\goth g}$.

\begin{prop}\label{psc2}
Let $\bk g{}$ be the orthogonal complement of $\bi g{}$ in $\tk {\k}{\sgg g{}}{\goth g}$.

{\rm (i)} For $\varphi $ in $\tk {\k}{\sgg g{}}{\goth g}$, $\varphi $ is in
$\bk g{}$ if and only if $\varphi (x,y)$ is in $[x,V_{x,y}]$ for all $(x,y)$ in 
a nonempty open subset of $\gg g{}$.

{\rm (ii)} The module $\bk g{}$ is free of rank $\b g{}-\rg$. Furthermore, the sequence
of maps
$$ (x,y) \mapsto [x,\varepsilon _{i}^{(1)}(x,y)],\ldots,
(x,y)\mapsto [x,\varepsilon _{i}^{(d_{i}-1)}(x,y)], \ i=1,\ldots,\rg$$
is a basis of $\bk g{}$.

{\rm (iii)} The orthogonal complement of $\bk g{}$ in $\tk {\k}{\sgg g{}}{\goth g}$ equals
$\bi g{}$.
\end{prop}

\begin{proof}
(i) Let $\varphi $ be in $\tk {\k}{\sgg g{}}{\goth g}$. If $\varphi $ is in
$\bk g{}$, then $\varphi (x,y)$ is orthogonal to $V_{x,y}$ for all $(x,y)$ in 
$\Omega _{{\goth g}}$. Then, according to Proposition~\ref{psc1},(iv), $\varphi (x,y)$
is in $[x,V_{x,y}]$ for all $(x,y)$ in $\Omega _{{\goth g}}$. Conversely,  
suppose that $\varphi (x,y)$ is in $[x,V_{x,y}]$ for all $(x,y)$ in a nonempty open 
subset $V$ of $\gg g{}$. By Proposition~\ref{psc1},(iv) again, for all $(x,y)$ in 
$V\cap \Omega _{{\goth g}}$, $\varphi (x,y)$ is orthogonal to the 
$\varepsilon _{i}^{(m)}(x,y)$'s, $i=1,\ldots,\rg$, $m=0,\ldots,d_{i}-1$, whence the 
assertion by Theorem~\ref{psc1}.

(ii) Let $\bkk $ be the submodule of $\tk {\k}{\sgg g{}}{\goth g}$ generated by the maps
$$ (x,y) \mapsto [x,\varepsilon _{i}^{(1)}(x,y)],\ldots,
(x,y)\mapsto [x,\varepsilon _{i}^{(d_{i}-1)}(x,y)], \ i=1,\ldots,\rg$$
According to (i), $\bkk $ is a submodule of $\bk g{}$. This module is free of rank 
$\b g{}-\rg$ since $[x,V_{x,y}]$ has dimension $\b g{}-\rg$ for all $(x,y)$ in 
$\Omega _{{\goth g}}$ by Proposition~\ref{psc1}, (iii) and (iv). According to (i), for 
$\varphi $ in $\bk g{}$, for all $(x,y)$ in $\Omega _{{\goth g}}$,
$$ \varphi (x,y) = \sum_{i=1}^{\rg} \sum_{m=1}^{d_{i}-1} 
a_{i,m}(x,y)[x,\varepsilon _{i}^{(m)}(x,y)]$$
with the $a_{i,m}$'s regular on $\Omega _{{\goth g}}$ and uniquely defined by this 
equality. Since $\Omega _{{\goth g}}$ is a big open subset of $\gg g{}$ and since 
$\gg g{}$ is normal, the $a_{i,m}$'s have a regular extension to $\gg g{}$. As a result,
$\varphi $ is in $\bkk$, whence the assertion.

(iii) Let $\varphi $ be in the orthogonal complement of $\bk g{}$ in 
$\tk {\k}{\sgg g{}}{\goth g}$. According to (ii), for all $(x,y)$ in 
$\Omega _{{\goth g}}$, $\varphi (x,y)$ is orthogonal to $[x,V_{x,y}]$. Hence by 
Proposition~\ref{psc1},(iv), $\varphi (x,y)$ is in $V_{x,y}$ for all $(x,y)$ in 
$\Omega _{{\goth g}}$. So, by Theorem~\ref{psc1}, $\varphi $ is in $\bi g{}$, whence
the assertion.
\end{proof}

Denote by ${\cal B}$ and ${\cal C}$ the localizations of $\bi g{}$ and $\bk g{}$ on 
$\gg g{}$ respectively. For $(x,y)$ in $\gg g{}$, let $C_{x,y}$ be the image of 
$\bk g{}$ by the evaluation map at $(x,y)$.

\begin{lemma}\label{lsc2}
There exists an affine open cover ${\cal O}$ of $\Omega _{{\goth g}}$ verifying the 
following condition: for all $O$ in ${\cal O}$, there exist some subspaces $E$ and 
$F$ of ${\goth g}$, depending on $O$, such that 
$$ {\goth g} = E\oplus V_{x,y} = F \oplus C_{x,y}$$
for all $(x,y)$ in $O$. Moreover, for all $(x,y)$ in $O$, the orthogonal complement 
of $V_{x,y}$ in ${\goth g}$ equals $C_{x,y}$.
\end{lemma}

\begin{proof}
According to Proposition~\ref{psc1},(iii) and (iv), for all $(x,y)$ in 
$\Omega _{{\goth g}}$, $V_{x,y}$ and $C_{x,y}$ have dimension $\b g{}$ and $\b g{}-\rg$
respectively so that the maps
$$ \Omega _{{\goth g}} \longrightarrow \ec {Gr}g{}{}{\b g{}}, \qquad 
(x,y) \longmapsto V_{x,y}, \qquad
\Omega _{{\goth g}} \longrightarrow \ec {Gr}g{}{}{\b g{}-\rg}, \qquad 
(x,y) \longmapsto C_{x,y}$$
are regular, whence the assertion.
\end{proof}

\section{Torsion and projective dimension} \label{tp}
Let $E$ and $E^{\#}$ be the quotients of $\tk {\k}{\sgg g{}}{\goth g}$ by $\bi g{}$ and 
$\bk g{}$ respectively. For $i$ positive integer, denote by $E_{i}$ the quotient of 
$\ex iE$ by its torsion module.

\subsection{} \label{tp1}
Let $\bi g{}^{*}$ and $\bk g{}^{*}$ be the duals of $\bi g{}$ and $\bk g{}$ respectively.

\begin{lemma}\label{ltp1}
{\rm (i)} The $\sgg g{}$-modules $E$ and $E^{\#}$ have projective dimension at most $1$.

{\rm (ii)} The $\sgg g{}$-modules $E$ and $E^{\#}$ are torsion free.

{\rm (iii)} The modules $\bk g{}$ and $\bi g{}$ are the duals of $E$ and $E^{\#}$ 
respectively.

{\rm (iv)} The canonical morphism from $E$ to $\bk g{}^{*}$ is an embedding.
\end{lemma} 

\begin{proof}
(i) By definition, the short sequences of $\sgg g{}$-modules,
$$ 0 \longrightarrow \bi g{} \longrightarrow \tk {\k}{\sgg g{}}{\goth g} 
\longrightarrow E \longrightarrow 0$$
$$ 0 \longrightarrow \bk g{} \longrightarrow \tk {\k}{\sgg g{}}{\goth g} 
\longrightarrow E^{\#} \longrightarrow 0$$
are exact. Hence $E$ and $E^{\#}$ have projective dimension at most $1$ since $\bi g{}$ 
and $\bk g{}$ are free modules by Theorem~\ref{tsc1} and Proposition~\ref{psc2},(ii). 

(ii) The module $E$ is torsion free by Theorem~\ref{tsc1},(ii). By definition, 
for $\varphi $ in $\tk {\k}{\sgg g{}}{\goth g}$, $\varphi $ is in $\bk g{}$ if 
$p\varphi $ is in $\bk g{}$ for some $p$ in $\sgg g{}\setminus \{0\}$, whence $E^{\#}$ 
is torsion free.

(iii) According to the exact sequences of (i), the dual of $E$ is the orthogonal 
complement of $\bi g{}$ in $\tk {\k}{\sgg g{}}{\goth g}$ and the dual of $E^{\#}$ is the 
orthogonal complement of $\bk g{}$ in $\tk {\k}{\sgg g{}}{\goth g}$, whence the assertion
since $\bk g{}$ is the orthogonal complement of $\bi g{}$ in 
$\tk {\k}{\sgg g{}}{\goth g}$ by definition and since $\bi g{}$ is the orthogonal 
complement of $\bk g{}$ in $\tk {\k}{\sgg g{}}{\goth g}$ by Proposition~\ref{psc2},(iii).

(iv) Let $\overline{\omega }$ be in the kernel of the canonical morphism from 
$E$ to $\bk g{}^{*}$. Let $\omega $ be a representative of $\overline{\omega }$ in 
$\tk {\k}{\sgg g{}}{\goth g}$. According to Proposition~\ref{psc2},(iii), $\bi g{}$ is 
the orthogonal complement of $\bk g{}$ in $\tk {\k}{\sgg g{}}{\goth g}$ so that $\omega $
is in $\bi g{}$, whence the assertion.
\end{proof}

Set: 
$$ \varepsilon = \wedge _{i=1}^{\rg} 
\poie {\varepsilon }i{ \wedge \cdots \wedge }{i}{}{}{}{(0)}{(d_{i}-1)}$$
and for $i$ positive integer, denote by $\theta _{i}$ the morphism 
$$\tk {\k}{\sgg g{}}\ex i{{\goth g}} \longrightarrow 
\ex i{{\goth g}}\wedge \ex {\b g{}}{\bi g{}}, \qquad 
\varphi \longmapsto \varphi \wedge \varepsilon .$$

\begin{prop}\label{ptp1}
Let $i$ be a positive integer. 

{\rm (i)} The morphism $\theta _{i}$ defines through the quotient
an isomorphism from $E_{i}$ onto $\ex i{{\goth g}}\wedge \ex {\b g{}}{\bi g{}}$.

{\rm (ii)} The short sequence of $\sgg g{}$-modules
$$ 0 \longrightarrow \tk {\sgg g{}}{\bi g{}}E_{i} \longrightarrow 
\tk {\k}{{\goth g}}E_{i} \longrightarrow \tk {\sgg g{}}{E}E_{i} \longrightarrow 0$$
is exact.
\end{prop}

\begin{proof}
(i) For $j$ positive integer, denote by $\pi _{j}$ the canonical map from 
$\tk {\k}{\sgg g{}}\ex j{{\goth g}}$ to $\ex jE$. Let $\omega $ be in the kernel of 
$\pi _{i}$. Let $O$ be an element of the affine open cover of $\Omega _{{\goth g}}$ of 
Lemma~\ref{lsc2} and let $W$ be a subspace of ${\goth g}$ such that
$$ {\goth g} = W \oplus V_{x,y}$$
for all $(x,y)$ in $O$ so that $\pi _{1}$ induces an isomorphism 
$$ \tk {\k}{\k[O]}W \longrightarrow \tk {\sgg g{}}{\k[O]}E$$
Moreover, $\bi g{}$ is the kernel of $\pi _{1}$. Then, from the equality
$$ \tk {\k}{\k[O]}\ex i{{\goth g}} = \bigoplus _{j=0}^{i} 
\ex jW\wedge \tk {\sgg g{}}{\k[O]}\ex {i-j}{\bi g{}}$$
it results that the restriction of $\omega $ to $O$ is in 
$\tk {\sgg g{}}{\k[O]}\ex {i-1}{{\goth g}}\wedge \bi g{}$. Hence the restriction of 
$\omega \wedge \varepsilon $ to $O$ equals $0$ and $\omega $ is in the kernel of 
$\theta _{i}$ since $\ex i{{\goth g}}\wedge \ex {\b g{}}{\bi g{}}$ has no torsion as a 
submodule of a free module. As a result, $\theta _{i}$ defines through the quotient a 
morphism from $\ex iE$ to $\ex i{{\goth g}}\wedge \ex {\b g{}}{\bi g{}}$. Denote it 
by $\vartheta _{i}'$. Since $\ex i{{\goth g}}\wedge \ex {\b g{}}{\bi g{}}$ is torsion 
free, the torsion submodule of $\ex iE$ is contained in the kernel of $\vartheta '_{i}$. 
Hence $\vartheta '_{i}$ defines through the quotient a morphism from $E_{i}$ to $\ex i{{\goth g}}\wedge \ex {\b g{}}{\bi g{}}$. 
Denoting it by $\vartheta _{i}$, $\vartheta _{i}'$ and $\vartheta _{i}$ are surjective 
since too is $\theta _{i}$.

Let $\overline{\omega }$ be in the kernel of $\vartheta '_{i}$ and let $\omega $ be a 
representative of $\overline{\omega }$ in $\tk {\k}{\sgg g{}}\ex i{{\goth g}}$. Then 
$\omega \wedge \varepsilon =0$ so that the restriction of $\omega $ to the above open
subset $O$ is in $\tk {\sgg g{}}{\k[O]}\ex {i-1}{{\goth g}}\wedge \bi g{}$. As a result, 
the restriction of $\overline{\omega }$ to $O$ equals $0$. So, $\overline{\omega }$ is in 
the torsion submodule of $\ex iE$, whence the assertion.

(ii) By definition, the sequence 
$$ 0 \longrightarrow \bi g{} \longrightarrow \tk {\k}{\sgg g{}}{\goth g} 
\longrightarrow E \longrightarrow 0$$ 
is exact. Then the sequence
$$ {\mathrm {Tor}}_{1}^{\sgg g{}}(E,E_{i}) \longrightarrow \tk {\sgg g{}}{\bi g{}}E_{i}
\longrightarrow \tk {\k}{{\goth g}}E_{i} \longrightarrow \tk {\sgg g{}}{E}E_{i} 
\longrightarrow 0$$
is exact. By definition, $E_{i}$ is torsion free. As a result, 
$\tk {\sgg g{}}{\bi g{}}E_{i}$ is torsion free since $\bi g{}$ is a free module. Then, 
since ${{\mathrm {Tor}}}_{1}^{\sgg g{}}(E,E_{i})$ is a torsion module, its image in 
$\tk {\sgg g{}}{\bi g{}}E_{i}$ equals $0$, whence the assertion.
\end{proof}

\subsection{} \label{tp2}
For $i$ positive integer, $\dv ..$ has a canonical extension to 
$\tk {\k}{\sgg g{}}\ex i{{\goth g}}$ denoted again by $\dv ..$.

\begin{lemma}\label{ltp2}
Let $i$ be a positive integer. Let $T_{i}$ be the torsion module of 
$\tk {\sgg g{}}{E}E_{i}$ and let $T'_{i}$ be its inverse image by the canonical 
morphism $\tk {\k}{{\goth g}}E_{i}\rightarrow \tk {\sgg g{}}{E}E_{i}$.

{\rm (i)} The canonical morphism from $\ex iE$ to $\ex i{\bk g{}^{*}}$ defines through
the quotient an embedding of $E_{i}$ into $\ex i{\bk g{}^{*}}$.

{\rm (ii)} The module of $T'_{i}$ is the intersection of 
$\tk {\k}{{\goth g}}E_{i}$ and $\tk {\sgg g{}}{\bi g{}}{\ex i{\bk g{}^{*}}}$.

{\rm (iii)} The module $T'_{i}$ is isomorphic to 
${\mathrm {Hom}}_{\sgg g{}}(E^{\#},E_{i})$.
\end{lemma}

\begin{proof}
(i) According to Lemma~\ref{ltp1},(iii), there is a canonical morphism from 
$\ex iE$ to $\ex i{\bk g{}^{*}}$. Let $\overline{\omega }$ be in its kernel and let 
$\omega $ be a representative of $\overline{\omega }$ in 
$\tk {\k}{\sgg g{}}\ex i{{\goth g}}$. Then $\omega $ is orthogonal to $\ex i{\bk g{}}$
with respect to $\dv ..$. So for $O$ as in Lemma~\ref{lsc2}, the restriction 
of $\omega $ to $O$ is in $\tk {\sgg g{}}{\k[O]}\ex {i-1}{{\goth g}}\wedge \bi g{}$. 
Hence the restriction of $\overline{\omega }$ to $O$ equals $0$. In other words, 
$\overline{\omega }$ is in the torsion module of $\ex iE$, whence the assertion since 
$\ex i{\bk g{}^{*}}$ is a free module.

(ii) Since $\ex i{\bk g{}^{*}}$ is a free module, by Proposition~\ref{ptp1},(ii), 
there is a morphism of short exact sequences
$$\xymatrix{
0 \ar[r] & \tk {\sgg g{}}{\bi g{}}E_{i} \ar[r] \ar[d] & 
\tk {\k}{{\goth g}}E_{i} \ar[r] \ar[d] & \tk {\sgg g{}}{E}E_{i} \ar[r] \ar[d] & 0 
\\ 0 \ar[r] & \tk {\sgg g{}}{\bi g{}}{\ex i{\bk g{}^{*}}} \ar[r]  & 
\tk {\k}{{\goth g}}{\ex i{\bk g{}^{*}}} \ar[r]  & 
\tk {\sgg g{}}{E}{\ex i{\bk g{}^{*}}} \ar[r]  & 0 }$$ 
Moreover, the two first vertical arrows are embeddings. Hence $T'_{i}$ is the 
intersection of $\tk {\k}{{\goth g}}E_{i}$ and 
$\tk {\sgg g{}}{\bi g{}}{\ex i{\bk g{}^{*}}}$ in 
$\tk {\k}{{\goth g}}\ex i{\bk g{}^{*}}$.

(iii) According to the identification of ${\goth g}$ with its dual, 
$\tk {\k}{{\goth g}}E_{i}={\mathrm {Hom}}_{\k}({\goth g},E_{i})$. Moreover, according 
to the short exact sequence of $\sgg g{}$-modules
$$ 0 \longrightarrow \bk g{} \longrightarrow \tk {\k}{\sgg g{}}{\goth g} \longrightarrow 
E^{\#} \longrightarrow 0$$ 
the sequence of $\sgg g{}$-modules
$$ 0 \longrightarrow {\mathrm {Hom}}_{\sgg g{}}(E^{\#},E_{i}) \longrightarrow 
{\mathrm {Hom}}_{\k}({\goth g},E_{i}) \longrightarrow 
{\mathrm {Hom}}_{\sgg g{}}(\bk g{},E_{i}) \longrightarrow 
{\mathrm {Ext}}^{1}_{\sgg g{}}(E^{\#},E_{i})$$
is exact. For $\varphi $ in 
${\mathrm {Hom}}_{\sgg g{}}(\tk {\k}{\sgg g{}}{\goth g},E_{i})$, $\varphi $ is 
in the kernel of the third arrow if and only if $\bk g{}$ is contained in the kernel 
of $\varphi $. On the other hand, according to the identification of ${\goth g}$ and its
dual, ${\mathrm {Hom}}_{\k}({\goth g},\ex i{\bk g{}^{*}}$ identifies with 
$\tk {\k}{{\goth g}}\ex i{\bk g{}^{*}}$. By Proposition~\ref{psc2}, $\bi g{}$ is the 
orthogonal complement of $\bk g{}$ in $\tk {\k}{\sgg g{}}{\goth g}$ and 
$\bk g{}^{*}$ is a free $\sgg g{}$-module. So, for $\psi $ in 
${\mathrm {Hom}}_{\sgg g{}}(\tk {\k}{\sgg g{}}{\goth g},\ex i{\bk g{}^{*}})$, 
$\psi $ equals $0$ on $\bk g{}$ if and only if it is in 
$\tk {\sgg g{}}{\bi g{}}{\ex i{\bk g{}^{*}}}$, whence the assertion by (ii). 
\end{proof}

The following corollary results from Lemma~\ref{ltp2}.

\begin{coro}\label{ctp2}
Let $i$ be a positive integer and let $\overline{E_{i}}$ be the quotient of 
$\tk {\sgg g{}}E{E_{i}}$ by its torsion module. Then the short sequence of 
$\sgg g{}$-modules
$$ 0 \longrightarrow {\mathrm {Hom}}_{\sgg g{}}(E^{\#},E_{i}) \longrightarrow 
\tk {\k}{{\goth g}}E_{i} \longrightarrow \overline{E_{i}} \longrightarrow 0$$
is exact.
\end{coro}

\subsection{} \label{tp3}
Denote by ${\mathrm {Mod}}_{\sgg g{}}$ the category of finite $\sgg g{}$-modules. Let 
$\iota $ be the morphism 
$$\tk {\k}{\sgg g{}}{\goth g} \longrightarrow \bk g{}^{*}, \qquad 
v \longmapsto ( \mu \mapsto \dv v{\mu }) $$

\begin{lemma}\label{ltp3}
Let $\bj g{}$ be the quotient of $\bk g{}^{*}$ by $\iota (\tk {\k}{\sgg g{}}{\goth g})$.
Then the two functors $\tk {\sgg g{}}{\bj g{}}\bullet$ and 
${\mathrm {Ext}}^{1}_{\sgg g{}}(E^{\#},\bullet)$ of the category
${\mathrm {Mod}}_{\sgg g{}}$ are isomorphic.
\end{lemma}

\begin{proof}
For $d$ nonnegative integer, denote by ${\mathrm {Mod}}_{\sgg g{}}(d)$ the full
subcategory of ${\mathrm {Mod}}_{\sgg g{}}$ whose objects are the modules of projective 
dimension at most $d$. Prove by induction on $d$ that the restrictions to 
${\mathrm {Mod}}_{\sgg g{}}(d)$ of the functors
${\mathrm {Ext}}^{1}_{\sgg g{}}(E^{\#},\bullet)$ and $\tk {\sgg g{}}{\bj g{}}\bullet$
are isomorphic. Let $M$ be a finite $\sgg g{}$-module. Denoting by $d$ its projective 
dimension, there is a short exact sequence
$$ 0 \longrightarrow Z \longrightarrow P \longrightarrow M \longrightarrow 0$$
with $Z$ a module of projective dimension $d-1$ if $d>0$ and $Z=0$ otherwise. 

Suppose $d=0$. Then, from the short exact sequence
$$ 0 \longrightarrow \bk g{} \longrightarrow \tk {\k}{\sgg g{}}{\goth g} \longrightarrow 
E^{\#} \longrightarrow 0$$
one deduces the exact sequence
$$ 0 \longrightarrow {\mathrm {Hom}}_{\sgg g{}}(E^{\#},M) \longrightarrow 
{\mathrm {Hom}}_{\k}({\goth g},M) \longrightarrow 
{\mathrm {Hom}}_{\sgg g{}}(\bk g{},M) \longrightarrow 
{\mathrm {Ext}}^{1}_{\sgg g{}}(E^{\#},M) \longrightarrow 0 .$$
Since $\bk g{}$ is a free module, ${\mathrm {Hom}}_{\sgg g{}}(\bk g{},M)$ is functorially
isomorphic to $\tk {\sgg g{}}{\bk g{}^{*}}M$. Then by the right exactness of the functor
$\tk {\sgg g{}}{\bj g{}}\bullet$, there is an isomorphism of exact sequences
$$\xymatrix{{\mathrm {Hom}}_{\k}({\goth g},M) \ar[r] & 
{\mathrm {Hom}}_{\sgg g{}}(\bk g{},M) \ar[r] & 
{\mathrm {Ext}}^{1}_{\sgg g{}}(E^{\#},M) \ar[r] & 0 \\
\tk {\k}{{\goth g}}M \ar[r] \ar[u]^{\delta _{0}} & 
\tk {\sgg g{}}{\bk g{}^{*}}M \ar[r] \ar[u]^{\delta _{1}} &
\tk {\sgg g{}}{\bj g{}}M \ar[r] \ar[u]^{\delta } & 0 }$$
Since the two sequences depends functorially on $M$, from the isomorphisms of functors
$$ \tk {\k}{{\goth g}}\bullet \longrightarrow {\mathrm {Hom}}_{\k}({\goth g},\bullet),
\quad \tk {\sgg g{}}{\bk g{}^{*}}\bullet \longrightarrow 
{\mathrm {Hom}}_{\sgg g{}}(\bk g{},\bullet),$$
we deduce that the restrictions to ${\mathrm {Mod}}_{\sgg g{}}(0)$ of the functors 
${\mathrm {Ext}}^{1}_{\sgg g{}}(E^{\#},\bullet)$ and $\tk {\sgg g{}}{\bj g{}}\bullet$
are isomorphic. 

Suppose the statement true for $d-1$. Setting $Q := \iota (\tk {\k}{\sgg g{}}{\goth g})$,
one has two short exact sequences
$$ 0 \longrightarrow \bi g{} \longrightarrow \tk {\k}{\sgg g{}}{\goth g} \longrightarrow 
Q \longrightarrow 0, \qquad 
0 \longrightarrow Q \longrightarrow \bk g{}^{*} \longrightarrow \bj g{}
\longrightarrow 0 .$$
Since $E^{\#}$ has projective dimension at most $1$ by Lemma~\ref{ltp1}, 
${\mathrm {Ext}}^{2}_{\sgg g{}}(E^{\#},Z)=0$. Then, by induction hypothesis, one has a 
commutative diagram
$$\xymatrix{ & 0 & 0 & 0 & \\
& {\mathrm {Ext}}^{1}_{\sgg g{}}(E^{\#},Z) \ar[r]^{\dd} \ar[u] & 
{\mathrm {Ext}}^{1}_{\sgg g{}}(E^{\#},P) \ar[r]^{\dd} \ar[u] & 
{\mathrm {Ext}}^{1}_{\sgg g{}}(E^{\#},M) \ar[r]\ar@{.>}[u]   &  0 \\
0 \ar[r] & \tk {\sgg g{}}{\bk g{}^{*}}Z \ar[r]^{\dd} \ar[u]^{\delta } &
\tk {\sgg g{}}{\bk g{}^{*}}P \ar[r]^{\dd} \ar[u]^{\delta } &
\tk {\sgg g{}}{\bk g{}^{*}}M \ar[r] \ar@{.>}[u]^{\delta _{M}} & 0 \\
& \tk {\sgg g{}}{Q}Z \ar[r]^{\dd} \ar[u]^{\delta } &
\tk {\sgg g{}}{Q}P \ar[r]^{\dd} \ar[u]^{\delta } &
\tk {\sgg g{}}{Q}M \ar[r] \ar[u]^{\delta } & 0 } $$
with exact lines and columns since $\bk g{}^{*}$ is a free module. Let $a$ and $a'$ be in
$\tk {\sgg g{}}{\bk g{}^{*}}P$ such that $\dd a = \dd a'$. Then $a-a'=\dd a_{1}$ with 
$a_{1}$ in $\tk {\sgg g{}}{\bk g{}^{*}}Z$ so that 
$$ \dd \rond \delta a - \dd \rond \delta a' =  \dd \rond \delta \rond \dd a_{1} = 0$$  
whence a morphism 
$$\tk {\sgg g{}}{\bk g{}^{*}}M  \stackrel{\delta _{M}}\longrightarrow  
{\mathrm {Ext}}^{1}_{\sgg g{}}(E^{\#},M) $$
uniquely defined by the equality $\delta _{M}\rond \dd = \dd \rond \delta $.

Let $a$ be in ${\mathrm {Ext}}^{1}_{\sgg g{}}(E^{\#},M)$. Then 
$$ a = \dd \rond \delta a_{1} = \delta _{M}\rond \dd a_{1} \quad  \text{with} \quad
a_{1} \in \tk {\sgg g{}}{\bk g{}^{*}}P .$$
Hence $\delta _{M}$ is surjective. Let $b$ be in the kernel of $\delta _{M}$. Then
$$ b = \dd b_{1}, \quad \dd \rond \delta b_{1} = 0, \quad 
\delta b_{1} = \dd \rond \delta b_{2} \quad  \text{with} \quad 
b_{1} \in \tk {\sgg g{}}{\bk g{}^{*}}P, \quad  b_{2} \in \tk {\sgg g{}}{\bk g{}^{*}}Z .$$
so that $b_{1} -\dd b_{2} = \delta b_{3}$ with $b_{3}$ in $\tk {\sgg g{}}QP$, whence
$b=\delta \rond \dd b_{3}$. As a result, the above diagram is canonically completed
by an exact third column and one has an isomorphism of short exact sequences 
$$\xymatrix{{\mathrm {Ext}}^{1}_{\sgg g{}}(E^{\#},Z) \ar[r] & 
{\mathrm {Ext}}^{1}_{\sgg g{}}(E^{\#},P) \ar[r] & 
{\mathrm {Ext}}^{1}_{\sgg g{}}(E^{\#},M) \ar[r] & 0 \\
\tk {\sgg g{}}{\bj g{}}Z \ar[r] \ar[u] & \tk {\sgg g{}}{\bj g{}}P \ar[r] \ar[u] &
\tk {\sgg g{}}{\bj g{}}M \ar[r] \ar[u] & 0 }.$$
Since the two sequences depends functorially on the short exact sequence
$$ 0 \longrightarrow Z \longrightarrow P \longrightarrow M \longrightarrow 0$$
and since the restrictions to ${\mathrm {Mod}}_{\sgg g{}}(d-1)$ of the two functors
${\mathrm {Ext}}^{1}_{\sgg g{}}(E^{\#},\bullet)$ and $\tk {\sgg g{}}{\bj g{}}\bullet$
are isomorphic, too is their restrictions to ${\mathrm {Mod}}_{\sgg g{}}(d)$, 
whence the lemma since all object of ${\mathrm {Mod}}_{\sgg g{}}$ has a finite projective
dimension.
\end{proof}

From the exact sequence,
$$ 0 \longrightarrow \bi g{} \longrightarrow \tk {\k}{\sgg g{}}{\goth g} 
\longrightarrow \bk g{}^{*} \longrightarrow \bj g{} \longrightarrow 0 ,$$
we deduce the graded homology complex,
$$ 0 \longrightarrow \tk {\sgg g{}}{\bi g{}}E_{i} \longrightarrow 
\tk {\k}{{\goth g}}E_{i} \longrightarrow \tk {\sgg g{}}{\bk g{}^{*}}E_{i}
\longrightarrow \tk {\sgg g{}}{\bj g{}}E_{i} \longrightarrow 0$$
denoted by $C_{\bullet}$. For $i$ positive integer, let $d_{i}$ and $d'_{i}$ be the 
projective dimensions of $E_{i}$ and ${\mathrm {Hom}}_{\sgg g{}}(E^{\#},E_{i})$. 

\begin{lemma}\label{l2tp3}
Let $Q$ be the space of cycles of degree $2$ of the complex $C_{\bullet}$.

{\rm (i)} Denoting by $d''_{i}$ the projective dimension of 
${\mathrm {Ext}}^{1}_{\sgg g{}}(E^{\#},E_{i})$, $d'_{i}$ is at most 
$\sup \{d''_{i}-2,d_{i}\}$.

{\rm (ii)} The complex $C_{\bullet}$ has no homology in degree $0$, $1$ and $3$. 
Moreover, $Q$ identifies with ${\mathrm {Hom}}_{\sgg g{}}(E^{\#},E_{i})$.

{\rm (iii)} The module ${\mathrm {Hom}}_{\sgg g{}}(E^{\#},E_{i})$ has projective 
dimension at most $d_{i}$.
\end{lemma}

\begin{proof}
(i) From the short exact sequence
$$ 0 \longrightarrow \bk g{} \longrightarrow \tk {\k}{\sgg g{}}{{\goth g}} 
\longrightarrow E^{\#} \longrightarrow 0$$
one deduces the exact sequence
$$ 0 \longrightarrow {\mathrm {Hom}}_{\sgg g{}}(E^{\#},E_{i}) \longrightarrow 
{\mathrm {Hom}}_{\k}({\goth g},E_{i}) \longrightarrow 
{\mathrm {Hom}}_{\sgg g{}}(\bk g{},E_{_{i}}) \longrightarrow 
{\mathrm {Ext}}^{1}_{\sgg g{}}(E^{\#},E_{i}) \longrightarrow 0$$
whence the two short exact sequences
$$0 \longrightarrow {\mathrm {Hom}}_{\sgg g{}}(E^{\#},E_{i}) \longrightarrow 
{\mathrm {Hom}}_{\k}({\goth g},E_{i}) \longrightarrow Z \longrightarrow 0$$ 
$$0 \longrightarrow Z \longrightarrow 
{\mathrm {Hom}}_{\sgg g{}}(\bk g{},E_{_{i}}) \longrightarrow 
{\mathrm {Ext}}^{1}_{\sgg g{}}(E^{\#},E_{i}) \longrightarrow 0$$
with $Z$ the image of the arrow 
$$ {\mathrm {Hom}}_{\k}({\goth g},E_{i}) \longrightarrow 
{\mathrm {Hom}}_{\sgg g{}}(\bk g{},E_{_{i}}) $$
Denoting by $d$ the projective dimension of $Z$, one deduces the inequalities
$$d'_{i} \leq \sup \{d-1,d_{i}\}, \qquad d \leq \sup\{d''_{i}-1,d_{i}\} $$
since $\bk g{}$ is a free module, whence the assertion.

(ii) By right exactness of the functor $\tk {\sgg g{}}{\bullet}E_{i}$, $C_{\bullet}$
has no homology in degree $0$ and $1$. Moreover, its space of cycles of degree $3$ is a 
torsion submodule of $C_{3}$. Since $E_{i}$ is torsion free and since $\bi g{}$ is free, 
$C_{3}$ has no torsion. Hence $C_{\bullet}$ has no homology in degree $3$. According to
Lemma~\ref{ltp2},(ii) and (iii), ${\mathrm {Hom}}_{\sgg g{}}(E^{\#},E_{i})$ identifies
with a submodule of $\tk {\k}{{\goth g}}E_{i}$. According to these identifications, $Q$
is the space of morphisms from $\tk {\k}{\sgg g{}}{\goth g}$ to $E_{i}$, equal to $0$ on 
$\bk g{}$, that is $Q={\mathrm {Hom}}_{\sgg g{}}(E^{\#},E_{i})$.

(iii) By (ii), one has a short exact sequence
$$ 0 \longrightarrow {\mathrm {Hom}}_{\sgg g{}}(E^{\#},E_{i}) \longrightarrow 
\tk {\sgg g{}}{\bk g{}^{*}}E_{i} 
\longrightarrow \tk {\sgg g{}}{\bj g{}}E_{i} \longrightarrow 0 .$$
So, $\tk {\sgg g{}}{\bj g{}}E_{i}$ has projective dimension at most 
$\sup \{d'_{i}+1,d_{i}\}$. According to Lemma~\ref{ltp3}, $\tk {\sgg g{}}{\bj g{}}E_{i}$
and ${\mathrm {Ext}}^{1}_{\sgg g{}}(E^{\#},E_{i})$ are isomorphic. So by (i),
$$ d'_{i} \leq \sup \{d'_{i}-1,d_{i}\},$$
whence $d'_{i}\leq d_{i}$.
\end{proof}

The following corollary results from Corollary~\ref{ctp2} and Lemma~\ref{l2tp3},(iii), 
since $\bi g{}$ is free.

\begin{coro}\label{ctp3}
Let $i$ be a positive integer. Then $\overline{E_{i}}$ has projective dimension 
at most $d_{i}+1$.
\end{coro}

\subsection{} \label{tp4}
For $i$ a positive integer and for $M$ a $\sgg g{}$-module, let consider on 
$M^{\tens i}$ the canonical action of the symmetric group ${\goth S}_{i}$. For $\sigma $ 
in ${\goth S}_{i}$, denote by $\epsilonup (\sigma )$ its signature. Let 
$M^{\tens i}_{{\mathrm {sign}}}$ be the submodule of elements $a$ of $M^{\tens i}$ such 
that $\sigma .a=\epsilonup (\sigma )a$ for all $\sigma $ in ${\goth S}_{i}$ and let 
$\delta _{i}$ be the endomorphism of $M^{\tens i}$,
$$ a \longmapsto \delta _{i}(a) = \frac{1}{i!} \sum_{\sigma \in {\goth S}_{i}}
\epsilonup (\sigma )\sigma .a .$$ 
Then $\delta _{i}$ is a projection of $M^{\tens i}$ onto $M^{\tens i}_{{\mathrm {sign}}}$.

For $L$ submodule of $\bk g{}^{*}$, denote by $L_{i}$ the image of $L^{\tens i}$ by 
the canonical map from $L^{\tens i}$ to $(\bk g{}^{*})^{\tens i}$ and set 
$L_{i,{\mathrm {sign}}} := L_{i}\cap (\bk g{}^{*})^{\tens i}_{{\mathrm {sign}}}$. Let 
$\overline{\ex iL}$ be the quotient of $\ex iL$ by its torsion module.
For $i\geq 2$, identify ${\goth S}_{i-1}$ with the stabilizer of $i$ in 
${\goth S}_{i}$ and denote by $L_{i-1,{\mathrm {sign}},1}$ the submodule of 
elements $a$ of $L_{i}$ such that $\sigma .a = \epsilonup (\sigma )a$ for all 
$\sigma $ in ${\goth S}_{i-1}$.  

\begin{lemma}\label{ltp4}
Let $i$ be a positive integer and let $L$ be a submodule of $\bk g{}^{*}$.

{\rm (i)} The module $L_{i}$ is isomorphic to the quotient of $L^{\tens i}$ by its 
torsion module.

{\rm (ii)} The module $L_{i,{\mathrm {sign}}}$ is isomorphic to $\overline{\ex iL}$.

{\rm (iii)} For $i\geq 2$, the module $L_{i,{\mathrm {sign}}}$ is a direct factor of
$L_{i-1,{\mathrm {sign}},1}$.
 
{\rm (iv)} For $i\geq 2$, the module $L_{i-1,{\mathrm {sign}},1}$ is isomorphic to
the quotient of $\tk {\sgg g{}}{\overline{\ex {i-1}L}}L$ by its torsion module.
\end{lemma}

\begin{proof}
(i) Let $L_{1}$ and $L_{2}$ be submodules of a free module $F$ over $\sgg g{}$. From 
the short exact sequence
$$ 0 \longrightarrow L_{2} \longrightarrow F \longrightarrow F/L_{2} \longrightarrow 0$$
one deduces the exact sequence
$$ {\mathrm {Tor}}_{\sgg g{}}^{1}(L_{1},L_{2}) \longrightarrow \tk {\sgg g{}}{L_{1}}L_{2}
\longrightarrow \tk {\sgg g{}}{L_{1}}F \longrightarrow 
\tk {\sgg g{}}{L_{1}}(F/L_{2}) \longrightarrow 0 .$$ 
Since $F$ is free, $\tk {\sgg g{}}{L_{1}}F$ is torsion free. Hence the kernel of the 
second arrow is the torsion module of $\tk {\sgg g{}}{L_{1}}L_{2}$ since 
${\mathrm {Tor}}_{\sgg g{}}^{1}(L_{1},L_{2})$ is a torsion module, whence the assertion 
by induction on $i$.

(ii) There is a commutative diagram
$$\xymatrix{ L^{\tens i} \ar[r] \ar[d]_{\delta _{i}} & (\bk g{}^{*})^{\tens i} 
\ar[d]^{\delta _{i}} \\ L^{\tens i}_{{\mathrm {sign}}} \ar[r] & 
(\bk g{}^{*})^{\tens i}_{{\mathrm {sign}}}}$$
so that $L_{i,{\mathrm {sign}}}$ is the image of $L^{\tens i}_{{\mathrm {sign}}}$ by the 
canonical morphism $L^{\tens i} \longrightarrow (\bk g{}^{*})^{\tens i}$, whence a 
commutative diagram
$$\xymatrix{
L^{\tens i}_{{\mathrm {sign}}} \ar[r] \ar[d] & \ex iL \ar[d] \\
(\bk g{}^{*})^{\tens i}_{{\mathrm {sign}}} \ar[r] & \ex i{\bk g{}^{*}} }$$ 
According to (i), the kernel of the left down arrow is the torsion module of 
$L^{\tens i}_{{\mathrm {sign}}}$ so that the kernel of the right down arrow is 
the torsion module of $\ex iL$ since the horizontal arrows are isomorphisms. Moreover,
the image of $L_{i,{\mathrm {sign}}}$ in $\ex i{\bk g{}^{*}}$ is the image of 
$\ex iL$. Hence $\overline{\ex iL}$ is isomorphic to $L_{i,{\mathrm {sign}}}$.

(iii) Denote by $Q_{i}$ the kernel of the endomorphism $\delta _{i}$ of 
$(\bk g{}^{*})^{\tens i}$. Since $\delta _{i}$ is a projection onto 
$(\bk g{}^{*})^{\tens i}_{{\mathrm {sign}}}$ such that $\delta _{i}(L_{i})$ is contained 
in $L_{i,{\mathrm {sign}}}$, 
$$(\bk g{}^{*})^{\tens i} = (\bk g{}^{*})_{{\mathrm {sign}}}^{\tens i} \oplus Q_{i}, 
\qquad L_{i} = L_{i,{\mathrm {sign}}}\oplus Q_{i}\cap L_{i} $$
whence  
$$L_{i-1,{\mathrm {sign}},1} = L_{i,{\mathrm {sgn}}} \oplus 
Q_{i}\cap L_{i-1,{\mathrm {sign}},1}$$
since $L_{i,{\mathrm {sign}}}$ is a submodule of $L_{i-1,{\mathrm {sign}},1}$.

(iv) Let $L'_{i}$ be the image of $L_{i-1,{\mathrm {sign}},1}$ by the canonical morphism 
$(\bk g{}^{*})^{\tens i}\rightarrow \tk {\sgg g{}}{\ex {i-1}{\bk g{}^{*}}}\bk g{}^{*}$. 
Then $L'_{i}$ is contained in $\tk {\sgg g{}}{\ex {i-1}{\bk g{}^{*}}}L$ since 
$\tk {\sgg g{}}{\ex {i-1}{\bk g{}^{*}}}L$ is a submodule of 
$\tk {\sgg g{}}{\ex {i-1}{\bk g{}^{*}}}\bk g{}^{*}$. Moreover, the 
morphism $L_{i-1,{\mathrm {sign}},1}\rightarrow L'_{i}$ is an isomorphism since too is
the morphism 
$$\tk {\sgg g{}}{(\bk g{}^{*})^{\tens (i-1)}_{{\mathrm {sign}}}}\bk g{}^{*}
\longrightarrow \tk {\sgg g{}}{\ex {i-1}{\bk g{}^{*}}}\bk g{}^{*} .$$  
From (ii), it results the commutative diagram
$$\xymatrix{ \tk {\sgg g{}}{L^{\tens (i-1)}_{{\mathrm {sign}}}}L \ar[r] \ar[d] &
\tk {\sgg g{}}{\overline{\ex {i-1}L}}L \ar[d] \\
L_{i-1,{\mathrm {sign},1}} \ar[r] & L'_{i} }$$
with the right down arrow surjective. According to (i), the kernel of the left 
down arrow is the torsion module of 
$\tk {\sgg g{}}{L^{\tens (i-1)}_{{\mathrm {sign}}}}L$. Hence the kernel of 
the right down arrow is the torsion module of  
$\tk {\sgg g{}}{\overline{\ex {i-1}L}}L$, whence the assertion.
\end{proof}

\begin{prop}\label{ptp4}
Let $i$ be a positive integer. Then $E_{i}$ and 
$\ex i{{\goth g}}\wedge \ex {\b g{}}{\bi g{}}$ have projective dimension at most $i$.
\end{prop}

\begin{proof}
According to Proposition~\ref{ptp1},(i), the modules $E_{i}$ and 
$\ex i{{\goth g}}\wedge \ex {\b g{}}{\bi g{}}$ are isomorphic. Prove by induction on 
$i$ that $E_{i}$ has projective dimension at most $i$. By Lemma~\ref{ltp1},(i), it is 
true for $i=1$. Suppose that it is true for $i-1$. According to Corollary~\ref{ctp3}, 
$\overline{E_{i-1}}$ has projective dimension at most $i$. By Lemma~\ref{ltp4}, for 
$L=E$, $E_{i}$ is a direct factor of $\overline{E_{i-1}}$ since $E$ is a submodule of 
$\bk g{}^{*}$ by Lemma~\ref{ltp1},(iv) and since $E_{i}=\overline{\ex iE}$. Hence 
$E_{i}$ has projective dimension at most $i$.
\end{proof}

\section{Main results} \label{mr}
Let $I_{{\goth g}}$ be the ideal of $\sgg g{}$ generated by the functions 
$(x,y) \mapsto \dv v{[x,y]}$ with $v$ in ${\goth g}$. The nullvariety of $I_{{\goth g}}$
in $\gg g{}$ is ${\cal C}({\goth g})$. Let $\dd $ be the $\sgg g{}$-derivation of the 
algebra $\tk {\k}{\sgg g{}}\ex {}{{\goth g}}$ such that $\dd v$ is the function 
$(x,y)\mapsto \dv v{[x,y]}$ on $\gg g{}$ for all $v$ in ${\goth g}$. The gradation on 
$\ex {}{{\goth g}}$ induces a gradation on $\tk {\k}{\sgg g{}}\ex {}{{\goth g}}$ so that 
$\tk {\k}{\sgg g{}}\ex {}{{\goth g}}$ is a graded homology complex.

\begin{lemma}\label{lmr}
Denote by $C_{\bullet}({\goth g})$ the graded submodule 
$\ex {}{{\goth g}}\wedge \ex {\b g{}}{\bi g{}}$ of 
$\tk {\k}{\sgg g{}}\ex {}{{\goth g}}$.

{\rm (i)} The graded module $C_{\bullet}({\goth g})$ is a graded subcomplex of 
$\tk {\k}{\sgg g{}}\ex {}{{\goth g}}$.

{\rm (ii)} The support of the homology of $C_{\bullet}({\goth g})$ is contained in 
${\cal C}({\goth g})$.
\end{lemma}

\begin{proof}
(i) Set:
$$ \varepsilon := \wedge _{i=1}^{\rg} 
\poie {\varepsilon }i{ \wedge \cdots \wedge }{i}{}{}{}{(0)}{(d_{i}-1)} .$$
Then $C_{\bullet}({\goth g})$ is the ideal of
$\tk {\k}{\sgg g{}}\ex {}{{\goth g}}$ generated by $\varepsilon $ since 
$\poie {\varepsilon }i{,\ldots,}{i}{}{}{}{(0)}{(d_{i}-1)}$, $i=1,\ldots,\rg$ is a
basis of $\bi g{}$ by Theorem~\ref{tsc1},(i). According to Theorem~\ref{tsc1},(iii), 
for $i=1,\ldots,\rg$ and for $m=0,\ldots,d_{i}-1$, $\varepsilon _{i}^{(m)}$ is a cycle of 
the complex $\tk {\k}{\sgg g{}}\ex {}{{\goth g}}$. Hence too is $\varepsilon $ and 
$C_{\bullet}({\goth g})$ is a subcomplex of $\tk {\k}{\sgg g{}}\ex {}{{\goth g}}$ as
an ideal generated by a cycle.

(ii) Let $(x_{0},y_{0})$ be in $\gg g{}\setminus {\cal C}({\goth g})$ and let 
$v$ be in ${\goth g}$ such that $\dv v{[x_{0},y_{0}]}\neq 0$. For some affine 
open subset $O$ of $\gg g{}$, containing $(x_{0},y_{0})$, $\dv v{[x,y]}\neq 0$ for all
$(x,y)$ in $O$. Then $\dd v$ is an invertible element of $\k[O]$. For 
$c$ a cycle of $\tk {\sgg g{}}{\k[O]}C_{\bullet}({\goth g})$, 
$$\dd (v\wedge c) = (\dd v)c $$
so that $c$ is a boundary of $\tk {\sgg g{}}{\k[O]}C_{\bullet}({\goth g})$. 
\end{proof}

\begin{theo}\label{tmr}
{\rm (i)} The complex $C_{\bullet}({\goth g})$ has no homology in degree bigger than 
$\b g{}$.

{\rm (ii)} The ideal $I_{{\goth g}}$ has projective dimension $2n-1$.

{\rm (iii)} The algebra $\sgg g{}/I_{{\goth g}}$ is Cohen-Macaulay.

{\rm (iv)} The projective dimension of the module 
$\ex {n}{{\goth g}}\wedge \ex {\b g{}}{\bi g{}}$ equals $n$. 
\end{theo}

\begin{proof}
(i) Let $Z$ be the space of cycles of degree $\b g{}+1$ of $C_{\bullet}({\goth g})$.
Then we deduce from $C_{\bullet}({\goth g})$ the complex
$$ 0 \longrightarrow C_{2n+\rg}({\goth g)} \longrightarrow \cdots \longrightarrow 
C_{n+\rg+2}({\goth g}) \longrightarrow Z \longrightarrow 0 .$$
According to Lemma~\ref{lmr},(ii), the support of its homology is contained in 
${\cal C}_{{\goth g}}$. In particular, its codimension in $\gg g{}$ is 
$$ 4n + 2\rg - (2n+2\rg) = 2n =  n + n-1 +1$$
According to Proposition~\ref{ptp4}, for $i=n+\rg+2,\ldots,2n+\rg$, 
$C_{i}({\goth g})$ has projective dimension at most $n$. Hence, by 
Corollary~\ref{cp}, this complex is acyclic and $Z$ has projective dimension at most 
$2n-2$, whence the assertion.

(ii) and (iii) Since $\bi g{}$ is a free module of rank $\b g{}{}$, 
$\ex {\b g{}{}}{\bi g{}}$ is a free module of rank $1$. By definition, the short sequence
$$ 0 \longrightarrow Z \longrightarrow {\goth g}\wedge \ex {\b g{}}{\bi g{}}
\longrightarrow I_{{\goth g}}\ex {\b g{}{}}{\bi g{}} \longrightarrow 0$$ 
is exact, whence the short exact sequence
$$ 0 \longrightarrow Z \longrightarrow {\goth g}\wedge \ex {\b g{}}{\bi g{}}
\longrightarrow I_{{\goth g}} \longrightarrow 0 .$$ 
Moreover, by Proposition~\ref{ptp4}, ${\goth g}\wedge \ex {\b g{}}{\bi g{}}$ 
has projective dimension at most $1$. Then, by (i), $I_{{\goth g}}$ has projective 
dimension at most $2n-1$. As a result the $\sgg g{}$-module $\sgg g{}/I_{{\goth g}}$ has 
projective dimension at most $2n$. Then by Auslander-Buchsbaum's theorem
\cite[\S 3, \no 3, Th\'eor\`eme 1]{Bou1}, the depth of the graded $\sgg g{}$-module 
$\sgg g{}/I_{{\goth g}}$ is at least
$$ 4\b g{}-2\rg - 2n = 2\b g{} $$
so that, according to \cite[\S 1, \no 3, Proposition 4]{Bou1}, the depth of
the graded algebra $\sgg g{}/I_{{\goth g}}$ is at least $2\b g{}$. In other words, 
$\sgg g{}/I_{{\goth g}}$ is Cohen-Macaulay since it has dimension $2\b g{}$. Moreover, 
since the graded algebra $\sgg g{}/I_{{\goth g}}$ has depth $2\b g{}$, the 
graded $\sgg g{}$-module $\sgg g{}/I_{{\goth g}}$ has projective dimension $2n$. Hence 
$I_{{\goth g}}$ has projective dimension $2n-1$.

(iv) By (i), $I_{{\goth g}}$ has projective dimension $2n-1$. Hence, according to 
Proposition~\ref{ptp4} and according to (ii) and Corollary~\ref{cp}, 
$\ex {n}{{\goth g}}\wedge \ex {\b g{}}{\bi g{}}$ has projective dimension $n$.
\end{proof}

\begin{theo}\label{t2mr}
The subscheme of $\gg g{}$ defined by $I_{{\goth g}}$ is Cohen-Macaulay and normal. 
Furthermore, $I_{{\goth g}}$ is a prime ideal.
\end{theo}

\begin{proof}
According to Theorem~\ref{tmr},(iii), the subscheme of $\gg g{}$ defined by
$I_{{\goth g}}$ is Cohen-Macaulay. According to \cite[Theorem 1]{Po}, it
is smooth in codimension $1$. So by Serre's normality criterion \cite[\S 1, \no 10,
Th\'eor\`eme 4]{Bou1}, it is normal. In particular, it is reduced and $I_{{\goth g}}$
is radical. According to ~\cite{Ric}, ${\cal C}({\goth g})$ is irreducible. Hence 
$I_{{\goth g}}$ is a prime ideal.
\end{proof}

\appendix

\section{Projective dimension and cohomology} \label{p}
Recall in this section classical results. Let $X$ be a Cohen-Macaulay irreducible 
affine algebraic variety and let $S$ be a closed subset of codimension $p$ of $X$. Let 
$P_{\bullet}$ be a complex of finite projective $\k[X]$-modules whose lenght 
$l$ is finite and let $\varepsilon $ be an augmentation morphism of $P_{\bullet}$ whose 
image is $R$, whence an augmented complex of $\k[X]$-modules,
$$ 0 \longrightarrow P_{l} \longrightarrow
P_{l-1} \longrightarrow \cdots \longrightarrow P_{0} \stackrel{\varepsilon }
\longrightarrow R \longrightarrow 0 .$$
Denote by ${\cal P}_{\bullet}$, ${\cal R}$, ${\cal K}_{0}$ the localizations on $X$ 
of $P_{\bullet}$, $R$, the kernel of $\varepsilon $ respectively and denote by 
${\cal K}_{i}$ the kernel of the morphism 
${\cal P}_{i}\longrightarrow {\cal P}_{i-1}$ for $i$ positive integer.

\begin{lemma} \label{lp}
Suppose that $S$ contains the support of the homology of the augmented complex  
$P_{\bullet}$. 

{\rm (i)} For all positive integer $i<p-1$ and for all projective 
$\an X{}$-module ${\cal P}$, ${\rm H}^{i}(X\setminus S,{\cal P})$ equals zero.

{\rm (ii)} For all nonnegative integer $j\leq l$ and for all positive integer 
$i<p-j$, the cohomology group 
${\rm H}^{i}(X\setminus S,{\cal K}_{l-j})$ equals zero.
\end{lemma}

\begin{proof}
(i) Let $i<p-1$ be a positive integer. Since the functor H$^{i}(X\setminus S,\bullet )$ 
commutes with the direct sum, it suffices to prove
${\mathrm {H}}^{i}(X\setminus S,\an X{})=0$. Since $S$ is a closed subset of $X$, 
one has the relative cohomology long exact sequence
$$ \cdots \longrightarrow 
{\mathrm H}^{i}_{S}(X,\an X{}) \longrightarrow {\mathrm H}^{i}(X,\an X{})
\longrightarrow {\mathrm H}^{i}(X\setminus S,\an X{}) \longrightarrow 
{\mathrm H}^{i+1}_{S}(X,\an X{}) \longrightarrow \cdots .$$
Since $X$ is affine, ${\mathrm {H}}^{i}(X,\an X{})$ equals zero and 
${\mathrm {H}}^{i}(X\setminus S,\an X{})$ is
isomorphic to ${\mathrm {H}}_{S}^{i+1}(X,\an X{})$. Since $X$ is Cohen-Macaulay, the 
codimension $p$ of $S$ in $X$ equals the depth of its ideal of definition in $\k[X]$ 
\cite[Ch. 6, Theorem 17.4]{Mat}. Hence, according to ~\cite[Theorem 3.8]{Gro}, 
${\mathrm {H}}_{S}^{i+1}(X,\an X{})$ and  ${\mathrm {H}}^{i}(X\setminus S,\an X{})$ equal
zero since $i+1<p$. 

(ii) Let $j$ be a nonnegative integer. Since $S$ contains the support of the homology of 
the complex $P_{\bullet}$, for all nonnegative integer $j$, one has the short exact 
sequence of $\an {X\setminus S}{}$-modules
$$ 0 \longrightarrow 
{\cal K}_{j+1}\left \vert \right. _{X\setminus S} \longrightarrow {\cal P}_{j+1}\left
\vert \right. _{X\setminus S} \longrightarrow 
{\cal K}_{j} \left \vert \right. _{X\setminus S}\longrightarrow 0 $$
whence the long exact sequence of cohomology
$$ \cdots \longrightarrow {\mathrm H}^{i}(X\setminus S,{\cal P}_{j+1})
\longrightarrow {\mathrm H}^{i}(X\setminus S,{\cal K}_{j}) \longrightarrow 
{\mathrm H}^{i+1}(X\setminus S,{\cal K}_{j+1}) \longrightarrow 
{\mathrm H}^{i+1}(X\setminus S,{\cal P}_{j+1}) \longrightarrow \cdots .$$
Then, by (i), for $0<i<p-2$, the cohomology groups 
${\mathrm {H}}^{i}(X\setminus S,{\cal K}_{j})$ and 
${\mathrm {H}}^{i+1}(X\setminus S,{\cal K}_{j+1})$ are isomorphic 
since $P_{j+1}$ is a projective module. Since ${\cal P}_{i}=0$ for $i>l$, 
${\cal K}_{l-1}$ and ${\cal P}_{l}$ have isomorphic restrictions to $X\setminus S$. In 
particular, by (i), for $0<i<p-1$, 
${\mathrm {H}}^{i}(X\setminus S,{\cal K}_{l-1})$ equal zero.
Then, by induction on $j$, for $0<i<p-j$, 
${\mathrm {H}}^{i}(X\setminus S,{\cal K}_{l-j})$ equals zero.
\end{proof}

\begin{prop} \label{pp}
Let $R'$ be a $\k[X]$-module containing $R$. Suppose that the following conditions 
are verified:
\begin{list}{}{}
\item {\rm (1)} $p$ is at least $l+2$,
\item {\rm (2)} $X$ is normal,
\item {\rm (3)} $S$ contains the support of the homology of the augmented complex 
$P_{\bullet}$.
\end{list}

{\rm (i)} The complex $P_{\bullet}$ is a projective resolution of $R$ of length $l$.

{\rm (ii)} Suppose that  $R'$ is torsion free and that $S$ contains the support in $X$ of
$R'/R$. Then $R'=R$.
\end{prop}

\begin{proof}
(i) Let $j$ be a positive integer. One has to prove that 
${\mathrm {H}}^{0}(X,{\cal K}_{j})$ is the image of $P_{j+1}$. By Condition (3), the 
short sequence of $\an {X\setminus S}{}$-modules
$$ 0 \longrightarrow \left. {\cal K}_{j+1} \right \vert _{X\setminus S} \longrightarrow 
{\cal P}_{j+1} \left. \vert \right. _{X\setminus S} \longrightarrow 
\left. {\cal K}_{j} \right \vert _{X\setminus S}\longrightarrow 0 $$ 
is exact, whence the cohomology long exact sequence
$$ 0 \longrightarrow {\mathrm {H}}^{0}(X\setminus S,{\cal K}_{j+1}) \longrightarrow 
{\mathrm {H}}^{0}(X\setminus S,{\cal P}_{j+1})
\longrightarrow {\mathrm {H}}^{0}(X\setminus S,{\cal K}_{j}) \longrightarrow 
{\mathrm H}^{1}(X\setminus S,{\cal K}_{j+1}) \longrightarrow \cdots .$$ 
By Lemma~\ref{lp},(ii), ${\mathrm {H}}^{1}(X\setminus S,{\cal K}_{j+1})$ equals $0$ since
$1<p-l+j+1$, whence the short exact sequence 
$$ 0 \longrightarrow {\mathrm {H}}^{0}(X\setminus S,{\cal K}_{j+1}) \longrightarrow 
{\mathrm {H}}^{0}(X\setminus S,{\cal P}_{j+1})
\longrightarrow {\mathrm {H}}^{0}(X\setminus S,{\cal K}_{j}) \longrightarrow 0 .$$ 
Since the codimension of $S$ in $X$ is at least $2$ and since $X$ is irreducible and
normal, the restriction morphism from $P_{j+1}$ to 
${\mathrm {H}}^{0}(X\setminus S,{\cal P}_{j+1})$ is an isomorphism. Let $\varphi $ be in 
${\mathrm {H}}^{0}(X,{\cal K}_{j})$. Then there exists an element $\psi $ of 
$P_{j+1}$ whose image $\psi '$ in ${\mathrm {H}}^{0}(X,{\cal K}_{j})$ has the same 
restriction to $X\setminus S$ as $\varphi $. Since $P_{j}$ is a projective module and 
since $X$ is irreducible, $P_{j}$ is torsion free. Then $\varphi =\psi '$ since 
$\varphi -\psi '$ is a torsion element of $P_{j}$, whence the assertion.

(ii) Let ${\cal R}'$ be the localization of $R'$ on $X$. Arguing as in (i), since $S$ 
contains the support of $R'/R$ and since $1<p-l$, the short sequence
$$ 0 \longrightarrow {\mathrm {H}}^{0}(X\setminus S,{\cal K}_{0}) \longrightarrow 
{\mathrm {H}}^{0}(X\setminus S,{\cal P}_{0})
\longrightarrow {\mathrm {H}}^{0}(X\setminus S,{\cal R}') \longrightarrow 0 $$ 
is exact. Moreover, the restriction morphism from $P_{0}$ to 
${\mathrm {H}}^{0}(X/S,{\cal P}_{0})$ is an isomorphism since the codimension of $S$ in 
$X$ is at least $2$ and since $X$ is irreductible and normal. Let $\varphi $ be in $R'$. 
Then for some $\psi $ in $P_{0}$, $\varphi -\varepsilon (\psi )$ is a torsion element of 
$R'$. So $\varphi =\varepsilon (\psi )$ since $R'$ is torsion free, whence the assertion.
\end{proof}

\begin{coro} \label{cp}
Let $C_{\bullet}$ be a homology complex of finite $\k[X]$-modules whose 
length $l$ is finite and positive. For $j=0,\ldots,l$, denote by $Z_{j}$ the space
of cycles of degree $j$ of $C_{\bullet}$. Suppose that the following conditions 
are verified:
\begin{list}{}{}
\item {\rm (1)} $S$ contains the support of the homology of the complex
$C_{\bullet}$,
\item {\rm (2)} for all $i$, $C_{i}$ is a submodule of a free module,
\item {\rm (3)} for $i=1,\ldots,l$, $C_{i}$ has projective dimension at most $d$,
\item {\rm (4)} $X$ is normal and $l+d\leq p-1$.
\end{list}
Then $C_{\bullet}$ is acyclic and for $j=0,\ldots,l$, $Z_{j}$ has projective dimension
at most $l+d-j-1$.
\end{coro}

\begin{proof}
Prove by induction on $l-j$ that the complex
$$ 0 \longrightarrow C_{l} \longrightarrow \cdots \longrightarrow C_{j+1}
\longrightarrow Z_{j} \longrightarrow 0 $$
is acyclic and that $Z_{j}$ has projective dimension at most $l+d-j-1$. For 
$j=l$, $Z_{j}$ equals zero since $C_{l}$ is torsion free by Condition (2) and since 
$Z_{l}$ a submodule of $C_{l}$, supported by $S$ by Condition (1). Suppose 
$j\leq l-1$ and suppose the statement true for $j+1$. By Condition (3), $C_{j+1}$ has
a projective resolution $P_{\bullet}$ whose length is at most $d$ and whose terms are 
finitely generated. By induction hypothesis, $Z_{j+1}$ has a projective resolution 
$Q_{\bullet}$ whose length is at most $l+d-j-2$ and whose terms are finitely generated, 
whence an augmented complex $R_{\bullet}$ of projective modules whose length is 
$l+d-j-1$,
$$ 0 \longrightarrow Q_{l+d-j-2} \oplus P_{l+d-j-1} \longrightarrow 
\cdots  \longrightarrow Q_{0} \oplus P_{1}
\longrightarrow P_{0} \longrightarrow  Z_{j} \longrightarrow 0 .$$
Denoting by $\dd $ the differentials of $Q_{\bullet}$ and $P_{\bullet}$, the restriction 
to $Q_{i}\oplus P_{i+1}$ of the differential of $R_{\bullet}$ is the map
$$(x,y) \mapsto (\dd x,\dd y +(-1)^{i} \delta (x)) ,$$ 
with $\delta $ the map which results from the injection of $Z_{j+1}$ into 
$C_{j+1}$. Since $P_{\bullet}$ and $Q_{\bullet}$ are projective resolutions, the complex 
$R_{\bullet}$ is a complex of projective modules having no homology 
in positive degree. Hence the support of the homology of the augmented complex 
$R_{\bullet}$ is contained in $S$ by Condition (1). Then, by Proposition~\ref{pp} and 
Condition (4), $R_{\bullet}$ is a projective resolution of $Z_{j}$ of length 
$l+d-j-1$ since $Z_{j}$ is a submodule of a free module by Condition (2), whence the 
corollary since $Z_{0}=C_{0}$ by definition.
\end{proof}

\begin{rema}
Let ${\mathrm {D}}(X)$ be the bounded derived category of finite $\k[X]$-modules. For $E$
an object of ${\mathrm {D}}(X)$, denote by ${\mathrm {Supp}}(E)$ the union of the 
supports in $X$ of the homology modules ${\mathrm {H}}_{i}(E)$ of $E$. By definition, the 
{\it homological dimension} of $E$, written $\hd (E)$, is the smallest integer 
$s$ such that $E$ is quasi-isomorphic to a complex of projective $\k[X]$-modules of 
length $s$. If no such integer exists, $\hd (E)=\infty $. Since $X$ is
Cohen-Macaulay, according to~\cite[Ch. 6, Theorem 17.4]{Mat}, we have the following 
proposition:

\begin{prop}~\cite[Corollary 5.5]{BM}\label{p2p}
Let $E$ be a non trivial object of ${\mathrm {D}}(X)$. Then for all irreducible 
component $\Gamma $ of ${\mathrm {Supp}}(E)$,
$$ \dim X - \dim \Gamma \leq \hd (E) .$$ 
\end{prop}

Corollary~\ref{cp} is a little bit smilar to Proposition~\ref{p2p}. But 
it is not a consequence of Proposition~\ref{p2p} since its proof does not use the 
normality of $X$.
\end{rema}

\end{document}